\documentclass[12pt]{article}

\usepackage{amsmath, epsfig, cite}
\usepackage{amssymb}
\usepackage{amsfonts}
\usepackage{latexsym}
\usepackage{float}
\usepackage{color}
\usepackage{mathrsfs}

\newtheorem{thm}{Theorem}[section]

\newtheorem{cor}[thm]{Corollary}

\newtheorem{lem}[thm]{Lemma}

\newcommand{\pf}{\noindent{\it Proof} }

\numberwithin{equation}{section}

\newcommand{\qed}{{\hfill$\square$}\medskip}

\begin{document}
	\begin{center}
		{\large\bf  Two Families of $q$-Supercongruences from Watson's Transformation  }
	\end{center}
	\vskip 2mm \centerline{Wei-Wei Qi}
	
	\begin{center}
		{\footnotesize MOE-LCSM, School of Mathematics and Statistics, Hunan Normal University, Hunan 410081, P.R. China\\[5pt]
			{\tt wwqi2022@foxmail.com} \\[10pt]
		}
	\end{center}
	
	\vskip 0.7cm \noindent{\bf Abstract.} In this paper, we establish two families of $q$-supercongruences  modulo the third and fourth powers of a cyclotomic polynomial by  employing Watson's ${}_{8}\phi_7$ transformation, the creative microscoping  method introduced by Guo and Zudilin, and the Chinese remainder theorem for coprime polynomials.

	\vskip 2mm
	\noindent{\it MR Subject Classifications}: 33D15, 11A07, 11B65	
	
	\section{Introduction} 
For a positive integer $r$, and $a_i$, $b_i$ $\in$ $\mathbb{C}$ with $b_i \notin \{\cdots, -3,-2,-1\}$, the (generalized) hypergeometric series ${}_{r+1}F_{r}$ is defined by

\begin{align*}
	\begin{aligned}
		{}_{r+1}F_{r}
		\left[
		\begin{array}{*{7}{c} c}
			a_1, & a_2, & a_3, & \cdots , &  a_{r+1} \\
			& b_1,& b_2,  & \cdots ,& b_{r}
		\end{array}
		;z
		\right]:=\sum_{k=0}^{\infty}\frac{\left(a_1\right)_k\left(a_2\right)_k\left(a_3\right)_k\cdots\left(a_{r+1}\right)_k}{k!\left(b_1\right)_k\left(b_2\right)_k\cdots\left(b_{r}\right)_k}z^k,
	\end{aligned}	
\end{align*}	
where $(x)_0:=1$, $\left(x\right)_k:=x(x+1)\cdots(x+k-1)$ is the $Pochhammer$ $symbol$ (rising factorials). This series converges for $|z|<1$. Define the truncated hypergeometric series as follows:	
\begin{align*}
	\begin{aligned}
		{}_{r+1}F_{r}
		\left[
		\begin{array}{*{7}{c} c}
			a_1, & a_2, & a_3, & \cdots , &  a_{r+1} \\
			& b_1,& b_2,  & \cdots ,& b_{r}
		\end{array}
		;z
		\right]_n:=\sum_{k=0}^{n}\frac{\left(a_1\right)_k\left(a_2\right)_k\left(a_3\right)_k\cdots\left(a_{r+1}\right)_k}{k!\left(b_1\right)_k\left(b_2\right)_k\cdots\left(b_{r}\right)_k}z^k,
	\end{aligned}	
\end{align*}

	In $2016$, Deines et al. \cite{f0-1} defined the higher dimensional analogues of Legendre curves as follows:
\begin{align*}
	C_{n,r}: y^n=(x_1x_2\cdots x_{n-1})^{n-1}(1-x_1)\cdots(1-x_{n-1})(x_1-rx_2\cdots x_{n-1}).
\end{align*}		
Differing by at most a scalar multiple, the hypergeometric series
\begin{align*}
	\begin{aligned}
		{}_{n}F_{n-1}
		\left[
		\begin{array}{*{7}{c} c}
			\frac{i}{n}, & \frac{i}{n}, & \frac{i}{n}, & \cdots , &  \frac{i}{n} \\
			& 1,& 1,  & \cdots ,& 1
		\end{array}
		;r
		\right]
	\end{aligned}	
\end{align*}		
	for any integer $i$ with $1\leq i \leq n-1$, may be regarded as a period of $C_{n, r}$ on the condition that it is convergent.  Deines et al. \cite[Theorem 2]{f0-1} showed that the number of rational points on $C_{n,r}$ over finite field $\mathbb{F}_q$ can be expressed by using Gaussian hypergeometric functions. They also proved that for any integer $n\geq3$ and prime $p\equiv 1 \pmod{n}$,
\begin{align}
	\begin{aligned}
		{}_{n}F_{n-1}
		\left[
		\begin{array}{*{7}{c} c}
			\frac{n-1}{n}, & \frac{n-1}{n}, & \frac{n-1}{n}, & \cdots , &  \frac{n-1}{n} \\
			& 1,& 1,  & \cdots ,& 1
		\end{array}
		;r
		\right]_{p-1}\equiv -\Gamma_p\left(\frac{1}{n}\right)^n \pmod{p^2}, \label{in-1}
	\end{aligned}	
\end{align}		
where $\Gamma_p(x)$ denotes the $p$-adic $Gamma$ $function$. They also conjectured that \eqref{in-1} holds modulo $p^3$, which was later confirmed by Wang and Pan \cite{f0-2}. In $2023$, Guo \cite[Theorem 1.1]{f0-3} established a $q$-analogue of \eqref{in-1}. Wang and Pan \cite{f0-2} also proved that for any prime $p\equiv 1 \pmod{5}$, 
\begin{align}
	\begin{aligned}
		{}_{5}F_{4}
		\left[
		\begin{array}{*{7}{c} c}
			\frac{4}{5}, & \frac{4}{5}, & \frac{4}{5}, & \frac{4}{5} , &  \frac{4}{5} \\
			& 1,& 1,  & 1 ,& 1
		\end{array}
		;1
		\right]_{p-1}\equiv -\Gamma_p\left(\frac{1}{5}\right)^5 \pmod{p^3}, \label{in-2}
	\end{aligned}	
\end{align}			
Recently, Guo \cite[Theorem 1.2]{f0-4} presented a $q$-congruence related to \eqref{in-2}.	Deines et al. \cite{f0-1} mentioned the following supercongruence for truncated hypergeometric series when $p$ is a prime with $p\equiv 1 \pmod{5}$,
\begin{align}
	\begin{aligned}
		{}_{5}F_{4}
		\left[
		\begin{array}{*{7}{c} c}
			\frac{2}{5}, & \frac{2}{5}, & \frac{2}{5}, & \frac{2}{5} , &  \frac{2}{5} \\
			& 1,& 1,  & 1 ,& 1
		\end{array}
		;1
		\right]_{p-1}\equiv -\Gamma_p\left(\frac{1}{5}\right)^5\Gamma_p\left(\frac{2}{5}\right)^5 \pmod{p^4}. \label{in-3}
	\end{aligned}	
\end{align}	
Pan	et al. \cite[Theorem 6.3]{f0-5} proved that for any prime $p\equiv 2 \pmod{5}$,
\begin{align}
	\begin{aligned}
		{}_{5}F_{4}
		\left[
		\begin{array}{*{7}{c} c}
			\frac{2}{5}, & \frac{2}{5}, & \frac{2}{5}, & \frac{2}{5} , &  \frac{2}{5} \\
			& 1,& 1,  & 1 ,& 1
		\end{array}
		;1
		\right]_{p-1}\equiv -\frac{p}{10}\Gamma_p\left(\frac{1}{5}\right)^5\Gamma_p\left(\frac{2}{5}\right)^5 \pmod{p^6}. \label{in-4}
	\end{aligned}	
\end{align}		
Subsequently, Guo \cite[Theorem 1.4 and Theorem 1.7]{f0-6} derived the corresponding $q$-analogues of \eqref{in-3} and \eqref{in-4} via  Jackson's ${}_{8}\phi_{7}$ summation.
	
	The construction of $q$-analogues for congruences has become an active research topic, with a large number of achievements reported in the literature. For more relevant work, see  (\cite{f0-3}, \cite{f0-4}, \cite{f0-6},\cite{f0-7}, \cite{f0-8}, \cite{f0-9} and on so on ).
	
Inspired by the above work, we are led to consider whether the two ${}_{5}F_{4}$ hypergeometric series
\begin{align*}
	\begin{aligned}
		{}_{5}F_{4}
		\left[
		\begin{array}{*{7}{c} c}
			\frac{3}{5}, & \frac{3}{5}, & \frac{3}{5}, & \frac{3}{5} , &  \frac{3}{5} \\
			& 1,& 1,  & 1 ,& 1
		\end{array}
		;1
		\right] 
		\quad and \quad
		{}_{5}F_{4}
		\left[
		\begin{array}{*{7}{c} c}
			\frac{1}{5}, & \frac{1}{5}, & \frac{1}{5}, & \frac{1}{5} , &  \frac{1}{5} \\
			& 1,& 1,  & 1 ,& 1
		\end{array}
		;1
		\right] 
	\end{aligned}	
\end{align*}		
possess truncated $q$-analogues. The main purpose of this paper is to construct these  $q$-analogues.

Throughout the paper, for $n\in \mathbb{Z}^+$, the $q$-integer is defined as
\begin{align*}
	[n]=[n]_q=\frac{1-q^n}{1-q}=1+q+q^2+\cdots+q^{n-1}	\quad for \quad n\in \mathbb{Z}^+.
\end{align*}			
Fixing $q$ with $0<|q|<1$, the $q$-shifted factorial is defined by
\begin{align*}
(a;q)_{\infty}=\prod_{j=0}^{\infty}(1-aq^j) \quad and \quad (a;q)_n=\frac{(a;q)_\infty}{(aq^n;q)_\infty}		\quad for \quad n\in \mathbb{Z}.
\end{align*}	
For succinctness, we use the shorthand notation
\begin{align*}
	(a_1,a_2,a_3,\cdots, a_t;q)_n=(a_1;q)_n(a_2;q)_n(a_3;q)_n\cdots(a_t;q)_n		\quad for \quad n\in \mathbb{Z}\cup \{\infty\}.
\end{align*}	
In addition, the $n$-th cyclotomic polynomial is given by
\begin{align*}
	\Phi_n(q)=\prod_{\substack{1\le k \le n\\[3pt](n,k)=1}}(q-\zeta^k),
\end{align*}
where $\zeta$ denotes a primitive $n$th root of unity.

The rest of the paper is organized as follows.  We present our main results in the next section. In Section $3$, we  state three key Lemmas  by using Watson's ${}_{8}\phi_{7}$ transformation formula, the `creative microscoping' method introduced by Guo and Zudilin \cite{f0} and the Chinese remainder theorem for coprime polynomials.  The proofs of the main theorems are given in Section $4$.

\section{Main results}		
	\begin{thm}
		For a positive integer $n$ with  $n \equiv 3 \pmod{5}$, we have
		\begin{align}
			\begin{aligned}
			\sum_{k=0}^{(n-3)/5}\frac{(1+q^{10k+3})(q^6;q^{10})_k^5}{(1+q^3)(q^{10};q^{10})_k^5}q^{5k}&\equiv q^{6(3-n)/5}\frac{(q^{12},q^{16};q^{10})_{(n-3)/5}}{(q^6,q^{10};q^{10})_{(n-3)/5}} \left(1-[2n]^2\sum_{j=1}^{(n-3)/5}\frac{q^{10j}}{[10j]^2}\right)\\
			&\times \sum_{k=0}^{(n-3)/5}\frac{(q^{6},q^{6},q^{6},q^{7};q^{10})_kq^{10 k}}{(q^{10},q^{10},q^{12},q^{13};q^{10})_k} \pmod{\Phi_n(q)^3}. \label{th1}
		\end{aligned}	
	\end{align}	
	\end{thm}	
		
Let $n=p$ be a prime. Taking $q\rightarrow 1$ in \eqref{th1}, we get the following result.
	\begin{cor}
	For any odd prime $p$ with  $p \equiv 3 \pmod{5}$, there holds
	\begin{align*}
		\begin{aligned}
				{}_{5}F_{4}
			\left[
			\begin{array}{*{7}{c} c}
				\frac{3}{5}, & \frac{3}{5}, & \frac{3}{5}, & \frac{3}{5} , &  \frac{3}{5} \\
				& 1,& 1,  & 1 ,& 1
			\end{array}
			;1
			\right]_{(p-3)/5}&\equiv\frac{\left(\frac{6}{5}\right)_{(p-3)/5}\left(\frac{8}{5}\right)_{(p-3)/5}}{\left(1\right)_{(p-3)/5}\left(\frac{3}{5}\right)_{(p-3)/5}}\left(1-\frac{p^2}{25}\sum_{j=1}^{(p-3)/5}\frac{1}{j^2}\right)\\
			&\times\sum_{k=0}^{(p-3)/5}\frac{\left(\frac{3}{5}\right)_k^3\left(\frac{7}{10}\right)_k}
			{k!^2\left(\frac{6}{5}\right)_k\left(\frac{13}{10}\right)_k}  \pmod{p^3}. 
		\end{aligned}	
	\end{align*}	
\end{cor}		
	
\begin{thm}
	For a positive integer $n$ with  $n \equiv 4 \pmod{5}$, we have
	\begin{align}
		\begin{aligned}
			\sum_{k=0}^{(2n-3)/5}\frac{(1+q^{10k+3})(q^6;q^{10})_k^5}{(1+q^3)(q^{10};q^{10})_k^5}q^{5k}&\equiv q^{6(3-2n)/5}\frac{(1-q)[4n](q^{12};q^{10})_{(2n-3)/5}}{(1-q^6)(q^{10};q^{10})_{(2n-3)/5}} \left(1-[4n]^2\sum_{j=1}^{(2n-3)/5}\frac{q^{10j}}{[10j]^2}\right)\\
			&\times \sum_{k=0}^{(2n-3)/5}\frac{(q^{6},q^{6},q^{6},q^{7};q^{10})_kq^{10 k}}{(q^{10},q^{10},q^{12},q^{13};q^{10})_k} \pmod{\Phi_n(q)^4}. \label{th2}
		\end{aligned}	
	\end{align}	
\end{thm}	
			
For any prime  $n=p$, setting $q\rightarrow 1$ in \eqref{th2}, we obtain the following congruence.	
\begin{cor}
	For any odd prime $p$ with  $p \equiv 4 \pmod{5}$, we have
	\begin{align*}
		\begin{aligned}
			{}_{5}F_{4}
			\left[
			\begin{array}{*{7}{c} c}
				\frac{3}{5}, & \frac{3}{5}, & \frac{3}{5}, & \frac{3}{5} , &  \frac{3}{5} \\
				& 1,& 1,  & 1 ,& 1
			\end{array}
			;1
			\right]_{(2p-3)/5}&\equiv  \frac{2p\left(\frac{6}{5}\right)_{(2p-3)/5}}{3\left(1\right)_{(2p-3)/5}}\left(1-\frac{4p^2}{25}\sum_{j=1}^{(2p-3)/5}\frac{1}{j^2}\right)\\
			&\times\sum_{k=0}^{(2p-3)/5}\frac{\left(\frac{3}{5}\right)_k^3\left(\frac{7}{10}\right)_k}
			{k!^2\left(\frac{6}{5}\right)_k\left(\frac{13}{10}\right)_k}  \pmod{p^4}. 
		\end{aligned}	
	\end{align*}	
\end{cor}			

\begin{thm}
	For any  positive integer $n$ satisfying  $n \equiv 1 \pmod{5}$, we have
	\begin{align}
		\begin{aligned}
			\sum_{k=0}^{(n-1)/5}\frac{(1+q^{10k+1})(q^2;q^{10})_k^5}{(1+q)(q^{10};q^{10})_k^5}q^{15k}&\equiv q^{2(1-n)/5}\frac{(q^{4},q^{12};q^{10})_{(n-1)/5}}{(q^2,q^{10};q^{10})_{(n-1)/5}} \left(1-[2n]^2\sum_{j=1}^{(n-1)/5}\frac{q^{10j}}{[10j]^2}\right)\\
			&\times \sum_{k=0}^{(n-1)/5}\frac{(q^{2},q^{2},q^{2},q^{9};q^{10})_k}{(q^{4},q^{10},q^{10},q^{11};q^{10})_k}q^{10 k} \pmod{\Phi_n(q)^3}. \label{th-02}
		\end{aligned}	
	\end{align}	
\end{thm}
		
Let $n=p$ be a prime. Taking $q\rightarrow 1$ in \eqref{th-02},  we can obtain the following congruence.	
\begin{cor}
 For any odd prime $p\equiv 1 \pmod{5}$, we have
\begin{align*}
	\begin{aligned}
		{}_{5}F_{4}
		\left[
		\begin{array}{*{7}{c} c}
			\frac{1}{5}, & \frac{1}{5}, & \frac{1}{5}, & \frac{1}{5} , &  \frac{1}{5} \\
			& 1,& 1,  & 1 ,& 1
		\end{array}
		;1
		\right]_{(p-1)/5}&\equiv\frac{\left(\frac{2}{5}\right)_{(p-1)/5}\left(\frac{6}{5}\right)_{(p-1)/5}}{\left(1\right)_{(p-1)/5}\left(\frac{1}{5}\right)_{(p-1)/5}}\left(1-\frac{p^2}{25}\sum_{j=1}^{(p-1)/5}\frac{1}{j^2}\right)\\
		&\times\sum_{k=0}^{(p-1)/5}\frac{\left(\frac{1}{5}\right)_k^3\left(\frac{9}{10}\right)_k}
		{k!^2\left(\frac{2}{5}\right)_k\left(\frac{11}{10}\right)_k}  \pmod{p^3}. 
	\end{aligned}	
\end{align*}			
\end{cor}		
	
\begin{thm}
	For any  positive integer $n$ satisfying  $n \equiv 3 \pmod{5}$, we have
	\begin{align}
		\begin{aligned}
			\sum_{k=0}^{(2n-1)/5}\frac{(1+q^{10k+1})(q^2;q^{10})_k^5}{(1+q)(q^{10};q^{10})_k^5}q^{15k}&\equiv \frac{[4n]q^{(2-4n)/5}(q^{4};q^{10})_{(2n-1)/5}}{(1+q)(q^{10};q^{10})_{(2n-1)/5}} \left(1-[4n]^2\sum_{j=1}^{(2n-1)/5}\frac{q^{10j}}{[10j]^2}\right)\\
			&\times \sum_{k=0}^{(2n-1)/5}\frac{(q^{2},q^{2},q^{2},q^{9};q^{10})_k}{(q^{4},q^{10},q^{10},q^{11};q^{10})_k}q^{10 k} \pmod{\Phi_n(q)^4}. \label{th-03}
		\end{aligned}	
	\end{align}	
\end{thm}		
	
Letting $n=p$ be a prime, and taking $q\rightarrow 1$ in \eqref{th-03},  we can arrive at the following consequence. 	
\begin{cor}	
For odd prime $p\equiv 3 \pmod{5}$, we have	
\begin{align*}
	\begin{aligned}
		{}_{5}F_{4}
		\left[
		\begin{array}{*{7}{c} c}
			\frac{1}{5}, & \frac{1}{5}, & \frac{1}{5}, & \frac{1}{5} , &  \frac{1}{5} \\
			& 1,& 1,  & 1 ,& 1
		\end{array}
		;1
		\right]_{(2p-1)/5}&\equiv2p\frac{\left(\frac{2}{5}\right)_{(2p-1)/5}}{\left(1\right)_{(2p-1)/5}}\left(1-\frac{4p^2}{25}\sum_{j=1}^{(2p-1)/5}\frac{1}{j^2}\right)\\
		&\times\sum_{k=0}^{(2p-1)/5}\frac{\left(\frac{1}{5}\right)_k^3\left(\frac{9}{10}\right)_k}
		{k!^2\left(\frac{2}{5}\right)_k\left(\frac{11}{10}\right)_k}  \pmod{p^4}. 
	\end{aligned}	
\end{align*}		
\end{cor}

\section{Key Lemmas}
Recall that the basic hypergeometric ${}_{r+1}\phi_{r}$ (see \cite{f1}) is given by
	\begin{align*}
	\begin{aligned}	
		{}_{r+1}\phi_{r}
		\left[
		\begin{array}{*{7}{c} c}
			a_1, & a_2, & a_3, & \cdots , &  a_{r+1} \\
			& b_1,& b_2,  & \cdots ,& b_r
		\end{array}
		;\ q,\ z
		\right]
		=\sum_{k=0}^{\infty}\frac{(a_1,a_2,a_3\cdots, a_{r+1})}{(q,b_1,b_2\cdots,b_r)}z^k.
	\end{aligned}	
\end{align*}
Then Watson's ${}_{8}\phi_{7}$ transformation \cite[Appendix (III.18)]{f1} can be stated as follows:
	\begin{align}
		\begin{aligned}	
{}_{8}&\phi_{7}
\left[
\begin{array}{*{7}{c} c}
	a, & qa^{\frac{1}{2}}, & -qa^{\frac{1}{2}}, & b, & c, & d, & e, &  q^{-n} \\
	& a^{\frac{1}{2}},& -a^{\frac{1}{2}}, & aq/b, & aq/c, & aq/d, & aq/e,& aq^{n+1}
\end{array}
;\ q,\ \frac{a^2 q^{n+2}}{bcde}
\right]\\
&=
\frac{(aq, aq/de; q)_n}{(aq/d, aq/e; q)_n}
{}_4\phi_3\left[\begin{array}{*{7}{c} c}
	 aq/bc,& d, & e,& q^{-n} \\
	& aq/b, & aq/c, & deq^{-n}/a
\end{array}; q, q\right]. \label{ll-0}
\end{aligned}	
\end{align}

\begin{lem}
	 Let $n$ be an integer with $n \equiv 3 \pmod{5}$, and let $a$, $b$ be indeterminates. Then modulo $(1-abq^{2n})(a-q^{2n})(b-q^{2n})$,
	\begin{align}
		\begin{aligned}
		\sum_{k=0}^{(n-3)/5}&\frac{(1+q^{10k+3})(q^6/a,q^6/b,abq^6,q^6,q^6;q^{10})_k}{(1+q^3)(aq^{10},bq^{10},q^{10}/(ab),q^{10},q^{10};q^{10})_k}q^{5k}\\
		&\equiv \{\frac{(b-q^{2n})(ab^2-a^2b-1+abq^{2n})}{(a-b)(1-ab^2)}\frac{(q^{16},q^{4-2n};q^{10})_{(n-3)/5}}{(q^{10}b,q^{10-2n}/b;q^{10})_{(n-3)/5}}\\
		&+\frac{(a-q^{2n})(1-abq^{2n})}{(a-b)(1-ab^2)}\frac{(q^{16},q^{4-2n};q^{10})_{(n-3)/5}}{(q^{10}a,q^{10-2n}/a;q^{10})_{(n-3)/5}}\}\\
		&\times\sum_{k=0}^{(n-3)/5}\frac{(q^{7},q^{6}/a,q^6/b,q^6ab;q^{10})_k}{(q^{10},q^{10},q^{12},q^{13};q^{10})_k}q^{10k}. \label{ll-1}
	\end{aligned}	
\end{align}
\end{lem}		
\pf. For $a=q^{2n}$ or $q^{-2n}/b$, the left-hand side of \eqref{ll-1} becomes
\begin{align}
	\begin{aligned}
		\sum_{k=0}^{(n-3)/5}&\frac{(1+q^{10k+3})(bq^{6+2n},q^{6-2n},q^{6}/b,q^6,q^6;q^{10})_k}{(1+q^3)(q^{10+2n},a^{10-2n}/b,q^{10}b,q^{10},q^{10};q^{10})_k}q^{10k}\\
		&=
		{}_{8}\phi_{7}
		\left[
		\begin{array}{*{7}{c} c}
			q^6, & q^{13}, & -q^{13}, & q^3, & q^6, & q^6/b, & bq^{6+2n}, & q^{6-2n} \\
			&q^{3}, & -q^{3}, & q^{13}, & q^{10}, & q^{10}b, & q^{10-2n}/b, & q^{10+2n}
		\end{array}
		;\ q^{10},\ q^5
		\right], \label{ll-2}
	\end{aligned}	
\end{align}
 By \eqref{ll-0}, the right-hand side of \eqref{ll-2} can be simplified to
\begin{align*}
	\begin{aligned}
	&\frac{(q^{16},q^{4-2n};q^{10})_{(n-3)/5}}{(q^{10}b,q^{10-2n}/b;q^{10})_{(n-3)/5}}\sum_{k=0}^{(n-3)/5}\frac{(q^{6-2n},bq^{6+2n},q^{6}/b,q^6;q^{10})_k}{(q^{10},q^{10},q^{12},q^{13})_k}q^{10k}\\
	&=\frac{(q^{16},q^{4-2n};q^{10})_{(n-3)/5}}{(q^{10}b,q^{10-2n}/b;q^{10})_{(n-3)/5}}\sum_{k=0}^{(n-3)/5}\frac{(q^{6}/a,abq^{6},q^{6}/b,q^6;q^{10})_k}{(q^{10},q^{10},q^{12},q^{13})_k}q^{10k}.
	\end{aligned}	
\end{align*}
Given that the polynomial $a-q^{2n}$ is coprime to $1-abq^{2n}$, we deduce that the $q$-congruence:
modulo $(1-abq^{2n})(a-q^{2n})$,
	\begin{align}
	\begin{aligned}
		\sum_{k=0}^{(n-3)/5}&\frac{(1+q^{10k+3})(q^6/a,q^6/b,abq^6,q^6,q^6;q^{10})_k}{(1+q^3)(aq^{10},bq^{10},q^{10}/(ab),q^{10},q^{10};q^{10})_k}q^{5k}\\
		&\equiv \frac{(q^{16},q^{4-2n};q^{10})_{(n-3)/5}}{(q^{10}b,q^{10-2n}/b;q^{10})_{(n-3)/5}}\sum_{k=0}^{(n-3)/5}\frac{(q^{6}/a,abq^{6},q^{6}/b,q^6;q^{10})_k}{(q^{10},q^{10},q^{12},q^{13})_k}q^{10k}. \label{ll-3}
	\end{aligned}	
\end{align}
Observing that the left-hand side of \eqref{ll-3} is symmetric in $a$ and $b$, we conclude that, modulo $b-q^{2n}$,
	\begin{align}
	\begin{aligned}
		\sum_{k=0}^{(n-3)/5}&\frac{(1+q^{10k+3})(q^6/a,q^6/b,abq^6,q^6,q^6;q^{10})_k}{(1+q^3)(aq^{10},bq^{10},q^{10}/(ab),q^{10},q^{10};q^{10})_k}q^{5k}\\
		&\equiv \frac{(q^{16},q^{4-2n};q^{10})_{(n-3)/5}}{(q^{10}a,q^{10-2n}/a;q^{10})_{(n-3)/5}}\sum_{k=0}^{(n-3)/5}\frac{(q^{6}/a,abq^{6},q^{6}/b,q^6;q^{10})_k}{(q^{10},q^{10},q^{12},q^{13})_k}q^{10k}. \label{ll-4}
	\end{aligned}	
\end{align}
Moreover, it is easy to check that
\begin{align}
	\frac{(b-q^{2n})(ab^2-a^2b-1+abq^{2n})}{(a-b)(1-ab^2)}\equiv 1 \pmod{(1-abq^{2n})(a-q^{2n})}, \label{ll-5}
\end{align}
and
\begin{align}
	\frac{(a-q^{2n})(1-abq^{2n})}{(a-b)(1-ab^2)}\equiv 1 \pmod{(b-q^{2n})}. \label{ll-6}
\end{align}
Since $(a-q^{2n})(1-qbq^{2n})$ and $(b-q^{2n})$ are coprime polynomials in $q$, we apply the Chinese remainder theorem for polynomials, from \eqref{ll-3}--\eqref{ll-6}, we arrive at the desired $q$-congruence \eqref{ll-1}. \qed\\

\begin{lem} \cite[Lemma 2.1]{f2}
	For positive integer $d$, $m$ and $n$ with $m<n$, and an integer $l$ such that $dm\equiv -l \pmod{n}$, it holds for $0\leq k\leq m$,
	\begin{align*}
		\begin{aligned}
			\frac{(q^la;q^d)_{m-k}}{(q^d/a;q^d)_{m-k}}\equiv (-a)^{m-2k}\frac{(q^la;q^d)_k}{(q^d/a;q^d)_k}q^{m(dm-d+2l)/2+(d-l)k} \pmod{\Phi_n(q)}.
		\end{aligned}	
	\end{align*}
\end{lem}	

\begin{lem} 
	Let $n$ be a positive integer statisfying $n\equiv 4 \pmod{5}$, and  let $a$ and $b$ be indeterminates. Then
	\begin{align}
		\begin{aligned}
			\sum_{k=0}^{(2n-3)/5}\frac{(1+q^{10k+3})(q^6/a,q^6/b,abq^6,q^6,q^6;q^{10})_k}{(1+q^3)(aq^{10},bq^{10},q^{10}/(ab),q^{10},q^{10};q^{10})_k}q^{5k}\equiv 0 \pmod{\Phi_n(q)}. \label{le2-3}
		\end{aligned}	
	\end{align}
\end{lem}	
\pf. By setting $d=10$, $l=6$ and $m=(2n-1)/5$ in the $q$-congruences outlined in Lemma $2.2$, for $0\leq k \leq (2n-3)/5$, we obtain
\begin{align*}
	\begin{aligned}
	\frac{(aq^6;q^{10})_{(2n-3)/5-k}}{(q^{10}/a;q^{10})_{(2n-3)/5-k}}  \equiv (-a)^{(2n-3)/5-2k}\frac{(aq^6;q^{10})_k}{(q^{10}/a;q^{10})_k}q^{(2n-2)(2n-3)/5+4k} \pmod{\Phi_n(q)}.
	\end{aligned}	
\end{align*}
Applying the above $q$-congruence five times and using the fact that $q^n\equiv 1\pmod{\Phi_n(q)}$, we deduce that
\begin{align*}
	\begin{aligned}
		&\frac{(1+q^{10((2n-3)/5-k)+3})(q^6/a,q^6/b,abq^6,q^6,q^6;q^{10})_{(2n-3)/5-k}}{(1+q^3)(aq^{10},bq^{10},q^{10}/(ab),q^{10},q^{10})_{(2n-3)/5}}q^{5((2n-3)/5-k)}\\
		& \hspace{2em} \equiv - \frac{(1+q^{10k+3})(q^6/a,q^6/b,abq^6,q^6,q^6;q^{10})_{k}}{(1+q^3)(aq^{10},bq^{10},q^{10}/(ab),q^{10},q^{10})_k}q^{5k} \pmod{\Phi_n(q)}.
	\end{aligned}	
\end{align*}
Observe that, modulo $\Phi_n(q)$, the $k$-th and $(2n-3)/5-k$-th summands on the left-hand side of \eqref{le2-3} cancel each other for $0\leq k\leq (2n-3)/5$. Consequently, we arrive at \eqref{le2-3}. \qed

\begin{lem}
	Let $n$ be an integer with $n \equiv 4 \pmod{5}$, and let $a$, $b$ be indeterminates. Then modulo $\Phi_n(q)(1-abq^{4n})(a-q^{4n})(b-q^{4n})$,
	\begin{align}
		\begin{aligned}
			\sum_{k=0}^{(2n-3)/5}&\frac{(1+q^{10k+3})(q^6/a,q^6/b,abq^6,q^6,q^6;q^{10})_k}{(1+q^3)(aq^{10},bq^{10},q^{10}/(ab),q^{10},q^{10};q^{10})_k}q^{5k}\\
			&\equiv \{\frac{(b-q^{4n})(ab^2-a^2b-1+abq^{4n})}{(a-b)(1-ab^2)}\frac{(q^{16},q^{4-4n};q^{10})_{(2n-3)/5}}{(q^{10}b,q^{10-4n}/b;q^{10})_{(2n-3)/5}}\\
			&+\frac{(a-q^{4n})(1-abq^{4n})}{(a-b)(1-ab^2)}\frac{(q^{16},q^{4-4n};q^{10})_{(2n-3)/5}}{(q^{10}a,q^{10-4n}/a;q^{10})_{(2n-3)/5}}\}\\
			&\times\sum_{k=0}^{(2n-3)/5}\frac{(q^{7},q^{6}/a,q^6/b,q^6ab;q^{10})_k}{(q^{10},q^{10},q^{12},q^{13};q^{10})_k}q^{10k}.  \label{lm-2}
		\end{aligned}	
	\end{align}
\end{lem}	

\pf.  Set $a=q^{4n}$ or $a=q^{-4n}/b$ and apply \eqref{ll-0}. Then we have
\begin{align*}
	\begin{aligned}
		\sum_{k=0}^{(2n-3)/5}&\frac{(1+q^{10k+3})(bq^{6+2n},q^{6-2n},q^{6}/b,q^6,q^6;q^{10})_k}{(1+q^3)(q^{10+2n},a^{10-2n}/b,q^{10}b,q^{10},q^{10};q^{10})_k}q^{10k}\\
		&=
		{}_{8}\phi_{7}
		\left[
		\begin{array}{*{7}{c} c}
			q^6, & q^{13}, & -q^{13}, & q^3, & q^6, & q^6/b, & bq^{6+4n}, & q^{6-4n} \\
			&q^{3}, & -q^{3}, & q^{13}, & q^{10}, & q^{10}b, & q^{10-4n}/b, & q^{10+4n}
		\end{array}
		;\ q^{10},\ q^5
		\right]\\
		&=\frac{(q^{16},q^{4-4n};q^{10})_{(2n-3)/5}}{(q^{10}b,q^{10-4n}/b;q^{10})_{(2n-3)/5}}\sum_{k=0}^{(2n-3)/5}\frac{(q^{6}/a,abq^{6},q^{6}/b,q^6;q^{10})_k}{(q^{10},q^{10},q^{12},q^{13})_k}q^{10k}.
	\end{aligned}	
\end{align*}
Clearly, the polynomials $(1-abq^{4n})(a-q^{4n})$ and $(b-q^{4n})$ are coprime polynomials in $q$. Utilizing the relations \eqref{ll-5} and \eqref{ll-6} with $q^{2n}\mapsto q^{4n}$, together with the above identity and the Chinese remainder theorem for polynomials, we deduce that \eqref{lm-2} holds modulo $(1-abq^{4n})(a-q^{4n})(b-q^{4n})$. Meanwhile, note that $(q^{16};q^{10})_{(2n-3)/3}\equiv 0\pmod{\Phi_n(q)}$. It follows that the right-hand of \eqref{lm-2} vanishes modulo $\Phi_n(q)$. In view of \eqref{le2-3},  \eqref{lm-2}  is valid  modulo $\Phi_n(q)$ as well. Furthermore, since the polynomials $(1-abq^{4n})(a-q^{4n})(b-q^{4n})$ and $\Phi_n(q)$ are pairwise relatively prime, this completes the proof of \eqref{lm-2}. \qed

\section{Proofs of Theorems}
	
\textbf{Proof of Theorem 2.1}: As $1-q^{2n}$ contains $\Phi_n(q)$, set $b=1$ in \eqref{ll-1}, we then get the $q$-congruence below: modulo $\Phi_n(q)(1-aq^{2n})(a-q^{2n})$,
	\begin{align}
	\begin{aligned}
		\sum_{k=0}^{(n-3)/5}&\frac{(1+q^{10k+3})(q^6/a,q^6,aq^6,q^6,q^6;q^{10})_k}{(1+q^3)(aq^{10},q^{10},q^{10}/a,q^{10},q^{10};q^{10})_k}q^{5k}\\
		&\equiv \{\frac{(1-q^{2n})(1-a+a^2-aq^{2n})}{(1-a)^2}\frac{(q^{16},q^{4-2n};q^{10})_{(n-3)/5}}{(q^{10},q^{10-2n};q^{10})_{(n-3)/5}}\\
		&-\frac{(a-q^{2n})(1-aq^{2n})}{(1-a)^2}\frac{(q^{16},q^{4-2n};q^{10})_{(n-3)/5}}{(q^{10}a,q^{10-2n}/a;q^{10})_{(n-3)/5}}\}\\
		&\times\sum_{k=0}^{(n-3)/5}\frac{(q^{7},q^{6}/a,q^6,q^6a;q^{10})_k}{(q^{10},q^{10},q^{12},q^{13};q^{10})_k}q^{10k}\\
		&=\frac{(q^{16},q^{4-2n};q^{10})_{(n-3)/5}}{(q^{10},q^{10-2n};q^{10})_{(n-3)/5}}\sum_{k=0}^{(n-3)/5}\frac{(q^{7},q^{6}/a,q^6,q^6a;q^{10})_k}{(q^{10},q^{10},q^{12},q^{13};q^{10})_k}q^{10k}\\
		&+\frac{(1-aq^{2n})(a-q^{2n})}{(1-a)^2}\sum_{k=0}^{(n-3)/5}\frac{(q^{7},q^{6}/a,q^6,q^6a;q^{10})_k}{(q^{10},q^{10},q^{12},q^{13};q^{10})_k}q^{10k}\\
		&\times\{\frac{(q^{16},q^{4-2n};q^{10})_{(n-3)/5}}{(q^{10},q^{10-2n};q^{10})_{(n-3)/5}}-\frac{(q^{16},q^{4-2n};q^{10})_{(n-3)/5}}{(q^{10}a,q^{10-2n}/a;q^{10})_{(n-3)/5}}\} \label{ll-5-0}
	\end{aligned}	
\end{align}
It  is clear that
\begin{align}
	\begin{aligned}
		&\frac{(q^{16},q^{4-2n};q^{10})_{(n-3)/5}}{(q^{10},q^{10-2n};q^{10})_{(n-3)/5}}-\frac{(q^{16},q^{4-2n};q^{10})_{(n-3)/5}}{(q^{10}a,q^{10-2n}/a;q^{10})_{(n-3)/5}}\\
		&\hspace{2em}\equiv \frac{(q^{16},q^{4};q^{10})_{(n-3)/5}}{(q^{10},q^{10};q^{10})_{(n-3)/5}}-\frac{(q^{16},q^{4};q^{10})_{(n-3)/5}}{(q^{10}a,q^{10}/a;q^{10})_{(n-3)/5}} \pmod{\Phi_n(q)}. \label{ll-6-0}
	\end{aligned}	
\end{align}
By the L'H\^opital rule, there holds
\begin{align}
	\begin{aligned}
		\lim\limits_{a\to 1} &\frac{(1-aq^{2n})(a-q^{2n})}{(1-a)^2}\{\frac{(q^{16},q^{4};q^{10})_{(n-3)/5}}{(q^{10},q^{10};q^{10})_{(n-3)/5}}-\frac{(q^{16},q^{4};q^{10})_{(n-3)/5}}{(q^{10}a,q^{10}/a;q^{10})_{(n-3)/5}}\}\\
		&=-(1-q^{2n})^2\frac{(q^{16},q^4;q^{10})_{(n-3)/5}}{(q^{10},q^{10};q^{10})_{(n-3)/5}}\sum_{j=1}^{(n-3)/5}\frac{q^{10j}}{(1-q^{10j})^2}. \label{ll-7}
	\end{aligned}	
\end{align}
Meanwhile, we have
\begin{align}
	\begin{aligned}
	\frac{(q^{4};q^{10})_{(n-3)/5}}{(q^{10};q^{10})_{(n-3)/5}}&\equiv	\frac{(q^{4-2n};q^{10})_{(n-3)/5}}{(q^{10-2n};q^{10})_{(n-3)/5}} \pmod{\Phi_n(q)}\\
		&=q^{(3-n)6/5}\frac{(q^{10};q^{10})_{(n-3)/5}}{(q^6;q^{10})_{(n-3)/5}}. \label{ll-8}
	\end{aligned}	
\end{align}
Thus, combining \eqref{ll-5-0}--\eqref{ll-8}, we arrive at the desired consequence \eqref{th-02}.\\

\noindent\textbf{Proof of Theorem 2.3}: Letting $b=1$ in \eqref{lm-2}, we acquire the following $q$-congruence: modulo $\Phi_n(q)(1-aq^{4n})(a-q^{4n})(1-q^{4n})$,
\begin{align}
	\begin{aligned}
		\sum_{k=0}^{(2n-3)/5}&\frac{(1+q^{10k+3})(q^6/a,q^6,aq^6,q^6,q^6;q^{10})_k}{(1+q^3)(aq^{10},q^{10},q^{10}/a,q^{10},q^{10};q^{10})_k}q^{5k}\\
		&\equiv \frac{(q^{16},q^{4-4n};q^{10})_{(2n-3)/5}}{(q^{10},q^{10-4n};q^{10})_{(2n-3)/5}}\sum_{k=0}^{(2n-3)/5}\frac{(q^{7},q^{6}/a,q^6,q^6a;q^{10})_k}{(q^{10},q^{10},q^{12},q^{13};q^{10})_k}q^{10k}\\
		 &+\frac{(a-q^{4n})(1-aq^{4n})}{(1-a)^2}\sum_{k=0}^{(2n-3)/5}\frac{(q^{7},q^{6}/a,q^6/b,q^6ab;q^{10})_k}{(q^{10},q^{10},q^{12},q^{13};q^{10})_k}q^{10k}\\
		 &\times\{\frac{(q^{16},q^{4-4n};q^{10})_{(2n-3)/5}}{(q^{10},q^{10-4n};q^{10})_{(2n-3)/5}}
		+\frac{(q^{16},q^{4-4n};q^{10})_{(2n-3)/5}}{(q^{10}a,q^{10-4n}/a;q^{10})_{(2n-3)/5}}\}. \label{th-2-1}
	\end{aligned}	
\end{align}	
	It is easy to see that	
\begin{align}
	\begin{aligned}
	&\frac{(q^{16},q^{4-4n};q^{10})_{(2n-3)/5}}{(q^{10},q^{10-4n};q^{10})_{(2n-3)/5}}
		+\frac{(q^{16},q^{4-4n};q^{10})_{(2n-3)/5}}{(q^{10}a,q^{10-4n}/a;q^{10})_{(2n-3)/5}}\\
		&\hspace{1em} \equiv \frac{(q^{16},q^4;q^{10})_{(2n-3)/5}}{(q^{10},q^{10};q^{10})_{(2n-3)/5}}
		+\frac{(q^{16},q^4;q^{10})_{(2n-3)/5}}{(q^{10}a,q^{10}/a;q^{10})_{(2n-3)/5}} \pmod{\Phi_n(q)}.\label{th-2-2}
	\end{aligned}	
\end{align}		
By 	the L'H\^opital rule, we obtain
\begin{align*}
	\begin{aligned}
		\lim\limits_{a\to 1} &\frac{(1-aq^{4n})(a-q^{4n})}{(1-a)^2}\{\frac{(q^{16},q^{4};q^{10})_{(2n-3)/5}}{(q^{10},q^{10};q^{10})_{(2n-3)/5}}-\frac{(q^{16},q^{4};q^{10})_{(2n-3)/5}}{(q^{10}a,q^{10}/a;q^{10})_{(2n-3)/5}}\}\\
		&=-(1-q^{4n})^2\frac{(q^{16},q^4;q^{10})_{(2n-3)/5}}{(q^{10},q^{10};q^{10})_{(2n-3)/5}}\sum_{j=1}^{(2n-3)/5}\frac{q^{10j}}{(1-q^{10j})^2}. 
	\end{aligned}	
\end{align*}	
Therefore, substituting \eqref{th-2-2} into \eqref{th-2-1}, taking limits as $a\rightarrow 1$, and combining the fact that	
\begin{align*}
	\begin{aligned}
		\frac{(q^4;q^{10})_{(2n-3)/5}}{(q^{10};q^{10})_{(2n-3)/5}}&\equiv \frac{(q^{4-4n};q^{10})_{(2n-3)/5}}{(q^{10-4n};q^{10})_{(2n-3)/5}} \pmod{\Phi_n(q)}\\
		& =q^{\frac{6(3-2n)}{5}}\frac{(q^{12};q^{10})_{(2n-3)/5}}{(q^{6};q^{10})_{(2n-3)/5}}.
	\end{aligned}	
\end{align*}		
we finish the proof of \eqref{th-02}.\\

Since the proofs of Theorem $2.5$ and $2.7$ are similar to those of Theorem $2.1$ and $2.3$, respectively, we only present their sketches of their proofs.\\

\noindent\textbf{Sketch of proof Theorem 2.5}: The proof is analogous to that of Theorem $2.1$. This time, We have a parametric version of \eqref{th-02}: modulo $(1-abq^{2n})(a-q^{2n})(b-q^{2n})$,
\begin{align}
	\begin{aligned}
		\sum_{k=0}^{(n-1)/5}&\frac{(1+q^{10k+1})(q^2/a,q^2/b,abq^2,q^2,q^2;q^{10})_k}{(1+q)(aq^{10},bq^{10},q^{10}/(ab),q^{10},q^{10};q^{10})_k}q^{15k}\\
		&\equiv \{\frac{(b-q^{2n})(ab^2-a^2b-1+abq^{2n})}{(a-b)(1-ab^2)}\frac{(q^{12},q^{8-2n};q^{10})_{(n-1)/5}}{(q^{10}b,q^{10-2n}/b;q^{10})_{(n-1)/5}}\\
		&+\frac{(a-q^{2n})(1-abq^{2n})}{(a-b)(1-ab^2)}\frac{(q^{12},q^{8-2n};q^{10})_{(n-1)/5}}{(q^{10}a,q^{10-2n}/a;q^{10})_{(n-1)/5}}\}\\
		&\times\sum_{k=0}^{(n-1)/5}\frac{(q^{9},q^{2}/a,q^2/b,q^2ab;q^{10})_k}{(q^{4},q^{10},q^{10},q^{11};q^{10})_k}q^{10k}.  \label{k1}
	\end{aligned}	
\end{align}
Upon setting $b=1$, it follows from \eqref{k1} that modulo $(1-aq^{2n})(a-q^{2n})(1-q^{2n})$
\begin{align}
	\begin{aligned}
		\sum_{k=0}^{(n-1)/5}&\frac{(1+q^{10k+1})(q^2/a,q^2/b,abq^2,q^2,q^2;q^{10})_k}{(1+q)(aq^{10},bq^{10},q^{10}/(ab),q^{10},q^{10};q^{10})_k}q^{15k}\\
		&\equiv \frac{(q^{12},q^{8-2n};q^{10})_{(n-1)/5}}{(q^{10},q^{10-2n};q^{10})_{(n-1)/5}}\sum_{k=0}^{(n-1)/5}\frac{(q^{9},q^{2}/a,q^2,q^2a;q^{10})_k}{(q^{4},q^{10},q^{10},q^{11};q^{10})_k}q^{10k}\\
		&+\frac{(1-aq^{2n})(a-q^{2n})}{(1-a)^2}\sum_{k=0}^{(n-1)/5}\frac{(q^{9},q^{2}/a,q^2/b,q^2ab;q^{10})_k}{(q^{4},q^{10},q^{10},q^{11};q^{10})_k}q^{10k}\\
		&\times\{\frac{(q^{12},q^{8-2n};q^{10})_{(n-1)/5}}{(q^{10},q^{10-2n};q^{10})_{(n-1)/5}}-
		\frac{(q^{12},q^{8-2n};q^{10})_{(n-1)/5}}{(q^{10}a,q^{10-2n}/a;q^{10})_{(n-1)/5}}\}. \label{ll-2-5} 
	\end{aligned}	
\end{align}
Then taking the limit as $a\rightarrow 1$ in \eqref{ll-2-5}, we arrive at \eqref{th-02}.\\

\noindent\textbf{Sketch of proof Theorem 2.7}: The argument for Theorem $2.7$ follows the same lines as  Theorem $2.3$.  We now state the corresponding congruence: Let $n$ be an integer with $n \equiv 3 \pmod{5}$, and let $a$, $b$ be indeterminates. Then modulo $\Phi_n(q)(1-abq^{4n})(a-q^{4n})(b-q^{4n})$,
	\begin{align}
		\begin{aligned}
			\sum_{k=0}^{(2n-1)/5}&\frac{(1+q^{10k+1})(q^2/a,q^2/b,abq^2,q^2,q^2;q^{10})_k}{(1+q)(aq^{10},bq^{10},q^{10}/(ab),q^{10},q^{10};q^{10})_k}q^{15k}\\
			&\equiv \{\frac{(b-q^{4n})(ab^2-a^2b-1+abq^{4n})}{(a-b)(1-ab^2)}\frac{(q^{12},q^{8-4n};q^{10})_{(2n-1)/5}}{(q^{10}b,q^{10-4n}/b;q^{10})_{(2n-1)/5}}\\
			&+\frac{(a-q^{4n})(1-abq^{4n})}{(a-b)(1-ab^2)}\frac{(q^{12},q^{8-4n};q^{10})_{(2n-1)/5}}{(q^{10}a,q^{10-4n}/a;q^{10})_{(2n-1)/5}}\}\\
			&\times\sum_{k=0}^{(2n-1)/5}\frac{(q^{9},q^{2}/a,q^2/b,q^2ab;q^{10})_k}{(q^{4},q^{10},q^{10},q^{11};q^{10})_k}q^{10k}.   \label{k2}
		\end{aligned}	
	\end{align}
Upon first setting $b=1$ and then letting $a\rightarrow 1$ in \eqref{k2}, we conclude that the   $q$-congruence \eqref{th-03} is valid.

\section*{Declarations}

\begin{flushleft}
	\textbf{Conflicts of Interest:}	The author declares that he has no conflict of interest.\\[13pt]
	
	\textbf{Data Availability Statement:} Not applicable.
\end{flushleft}

\end{document}